\documentclass[11pt,a4paper]{amsart}
\usepackage[T1]{fontenc}
\usepackage[utf8]{inputenc}
\usepackage{lmodern}
\usepackage{amsmath,amssymb,amsthm,mathtools}
\usepackage[a4paper,left=27mm,right=27mm,top=26mm,bottom=27mm]{geometry}
\usepackage{microtype}
\usepackage{needspace}
\usepackage{verbatim}
\usepackage{xcolor}
\usepackage[colorlinks=true,linkcolor=blue!45!black,citecolor=blue!45!black,urlcolor=blue!45!black]{hyperref}
\newtheorem{theorem}{Theorem}[section]
\newtheorem{lemma}[theorem]{Lemma}
\newtheorem{corollary}[theorem]{Corollary}
\newtheorem{proposition}[theorem]{Proposition}

\theoremstyle{definition}
\newtheorem{definition}[theorem]{Definition}
\newtheorem{remark}[theorem]{Remark}
\DeclareMathOperator{\supp}{supp}
\newcommand{\cA}{\mathcal{A}}

\newcommand{\An}{\cA_n}

\newcommand{\cF}{\mathcal{F}}
\newcommand{\cH}{\mathcal{H}}
\newcommand{\cT}{\mathcal{T}}

\allowdisplaybreaks[1]
\begin{document}
\hypersetup{pdftitle={Intersecting integer partitions: star bounds and counterexamples at every scale},pdfauthor={Yury Person and Thomas Schweser}}\title[Intersecting integer partitions]{Intersecting integer partitions:\\star bounds and counterexamples at every scale}
\author[Y.~Person]{Yury Person}
\address{Institut f\"ur Mathematik, Technische Universit\"at Ilmenau, 98684 Ilmenau, Germany} 
\email{yury.person@tu-ilmenau.de}

\author[T.~Schweser]{Thomas Schweser}
\address{Technische Hochschule Rosenheim, Fakult\"at f\"ur Angewandte Natur- und Geisteswissenschaften, Hochschulstra\ss e 1, 83024 Rosenheim, Germany} 
\email{thomas.schweser@th-rosenheim.de}
\begin{abstract}
Two integer partitions $t$-intersect if they have at least $t$ common parts, counted with multiplicity. 
We study the largest $t$-intersecting families of integer partitions of $n$ into exactly $k$ positive parts. The canonical $t$-star consists of the partitions containing at least $t$ ones. Applying Kupavskii's weak-spread theorem, we prove that this star is largest whenever $n\ge Ak^3$, for every fixed $A>24$ and all sufficiently large $k$, uniformly over $1\le t<k$. We also give counterexamples to Borg's conjecture at every intersection scale: for all sufficiently large $k$ and for every $1\le t\le \lceil k/4\rceil$, we describe  $t$-intersecting families on integer partitions of $n=\lfloor tk^2/3\rfloor$ into $k$ parts which are strictly larger than the corresponding canonical stars. 

Moreover we present a family of counterexamples to Borg's conjectures about largest $t$-intersecting families of integer partitions of $n$ without restricting the number of parts.
\end{abstract}
\maketitle

\section{Introduction}
A family $\cF\subseteq \binom{X}{k}$ of  $k$-sets of an $n$-element set $X$ is \emph{$t$-intersecting} if any two members of $\cF$ intersect in at least $t$ elements. The problem of determining the size of a largest $t$-intersecting family is \emph{one of the oldest problems in combinatorial extremal set theory}~\cite{AK1997} and was initiated by Erd\H{o}s, Ko and Rado~\cite{EKR1961}. The famous Erd\H{o}s--Ko--Rado theorem~\cite{EKR1961} states that  a $1$-intersecting (called simply \emph{intersecting}) family of $k$-subsets  of an $n$-element set has at most $\binom{n-1}{k-1}$ members when $n\ge 2k$.
A family consisting of all $k$-sets containing one fixed element, a so-called \emph{star}, is clearly intersecting and achieves this bound. More generally, the family of all $k$-sets containing a fixed $t$-set is a \emph{full $t$-star} and has size
$\binom{n-t}{k-t}$. The work of Frankl~\cite{Frankl1978} and subsequently Wilson~\cite{Wilson1984} established the sharp range $n\ge(t+1)(k-t+1)$ in which the bound $|\cF|\le \binom{n-t}{k-t}$ holds, for $1\le t<k$.
Ahlswede and Khachatrian's complete intersection theorem~\cite{AK1997} settled the remaining parameter ranges, thereby proving Frankl's conjecture~\cite{Frankl1978}
that, for some integer $i\ge0$, all $k$-sets containing at least $t+i$
elements of a fixed $(t+2i)$-set form a largest $t$-intersecting family.

Intersection problems have also been studied for multisets with fixed
cardinality over a finite ground set. Meagher and Purdy~\cite{MP2011}
determined the largest intersecting families in this setting, and
F\"uredi, Gerbner and Vizer~\cite{FGV2015} obtained $t$-intersection
results. For further results on intersecting families, see Borg's survey~\cite{Borg2012}.

Borg~\cite{Borg2014, Borg2016} initiated the study of intersecting \emph{integer partitions}. The theory of integer partitions is a classical topic in number theory; see, for example, the survey by Ahlgren and Ono~\cite{AO2001} for a general mathematical audience.
 An \emph{integer partition of $n$ into $k$ positive parts} is a nondecreasing sequence
$\lambda=(a_1,\ldots,a_k)$ with $\sum_i a_i=n$. We write $P_{n,k}$ for the set of these partitions and $p(n,k)=|P_{n,k}|$. Similarly, we write $P_n$ for the set $\cup_{k\in[n]} P_{n,k}$ of partitions on $n$ without restricting the number of parts and $p(n)$ for its cardinality $|P_n|$.  We count common parts  with multiplicity. Thus, if $m_\lambda(a)$ is the multiplicity of the value $a$, the intersection size of $\lambda$ and $\mu$ is
\[
 I(\lambda,\mu):=\sum_{a\ge1}\min\{m_\lambda(a),m_\mu(a)\}.
\]
A family is \emph{$t$-intersecting} if every two members have intersection size at least $t$.
A \emph{full $t$-star} is the family of all partitions in $P_{n,k}$ or in $P_n$ containing a fixed multiset of $t$ parts, counted with multiplicity.
The canonical $t$-star is
\[
 \mathcal T_{n,k,t}:=\{\lambda\in P_{n,k}:m_\lambda(1)\ge t\},
 \qquad |\mathcal T_{n,k,t}|=p(n-t,k-t),
\]
and similarly 
\[
 \mathcal{T}_{n,t}:=\{\lambda\in P_{n}:m_\lambda(1)\ge t\},
 \qquad |\mathcal{T}_{n,t}|=p(n-t).
\]

A family is \emph{trivial} if it is contained in a full $t$-star; otherwise it is \emph{nontrivial}.

 Integer partitions impose an additional fixed-sum condition.
Consequently, different part values play different roles, and the
ambient family is not invariant under arbitrary relabelling.
 Borg \cite{Borg2014} considered the related notion of \emph{strong
intersection}: two nondecreasing tuples strongly intersect if they
agree in at least one coordinate. He proved that the partitions
containing a part equal to $1$ form a largest strongly intersecting
subfamily of $P_{n,k}$ for every $2\le k\le n$.

For the intersection notion considered here, equal parts may occur in
different positions. 
Borg~\cite[Conjecture~1.6]{Borg2016} conjectured that $\mathcal T_{n,k,t}$ is largest for $1\le t<k\le n$, apart from $(n,k)=(t+7,t+2)$. The exceptional example begins with
\[
 \{(1,2,5),(1,3,4),(2,2,4),(2,3,3)\}\subseteq P_{8,3};
\]
it has four members, compared with $p(7,2)=3$ in the canonical star. Adding $t-1$ ones gives the stated exception. Borg proved the star bound for $n\le2k-t+1$ and for $n\ge3tk^5$.

We improve Borg's bound $n\ge3tk^5$ to  $n\ge Ak^3$, a range optimal up to a constant factor when $t=\Theta(k)$. Our proof uses the spread-approximation method of Kupavskii and Zakharov~\cite{KZ2024}, in the weak-spread form developed by Kupavskii \cite{Kupavskii2026} and exploits estimates on $p(n,k)$ proved by Szekeres~\cite{Szekeres1951,Szekeres1953} in the form given by Canfield~\cite{Canfield1997}. 

\Needspace{10\baselineskip}
\begin{theorem}[Cubic star range]\label{thm:main-upper}
For every fixed $A>24$, there is $k_0(A)$ such that, if $k\ge k_0(A)$, $n\ge Ak^3$, and $1\le t<k$, then every $t$-intersecting $\mathcal F\subseteq P_{n,k}$ satisfies
\[
 |\mathcal F|\le p(n-t,k-t).
\]
Equality holds if and only if $\mathcal F$ is a full $t$-star
of maximum size. If $k\ge t+3$, the canonical star
$\mathcal T_{n,k,t}$ is the unique maximum family.
If $A>48$ and $\mathcal F$ is nontrivial, then
\[
 |\mathcal F|\le\tfrac12p(n-t,k-t).
\]
\end{theorem}

The canonical star is therefore a largest family. The qualification about maximum stars matters: when $t=k-1$, every singleton family is largest. Other ties are possible when $k-t=2$.

We also obtain the following result which disproves Borg's conjecture from~\cite[Conjecture~1.6]{Borg2016} and provides a lower bound on $n$ below which a general statement like 
that of Theorem~\ref{thm:main-upper} cannot hold.
Write $\supp(\lambda)=\{a\ge1:m_\lambda(a)>0\}$ for the set of
distinct part values of $\lambda$.

\begin{theorem}[Counterexamples at every scale]
\label{thm:every-scale}
For every sufficiently large $k$ and every integer $1\le t\le \lceil k/4\rceil$, put $n=\lfloor tk^2/3\rfloor$. 
Then the family $\mathcal F_{n,k,t}$, defined by
\[
 \mathcal F_{n,k,t}:=\{\lambda\in P_{n,k}:|\operatorname{supp}(\lambda)\cap[t+2]|\ge t+1\}
\]
is $t$-intersecting and larger than the canonical $t$-star $\mathcal{T}_{n,k,t}$.
\end{theorem}

In particular we obtain $t$-intersecting families larger than canonical $t$-star for $n=\Theta(tk^2)$ with absolute implied constants. Thus, this covers every `scale' for $t$ from constant $1$-intersection to intersection $t$ proportional to the length $k$. We also provide small counterexamples of intersecting families with  $620$ and $626$ partitions in $P_{31,8}$, compared with $618$ in the canonical star. 

Regarding integer partitions without restrictions on the number of parts, Borg conjectured that the star $\mathcal{T}_{n,t}$  is largest for $t=1$ and, more generally, for every $t\ge1$~\cite[Conjectures~1.1 and~1.4]{Borg2016}. We prove that Borg's conjectures fail for all large $n$ and all $t\ge 1$ whenever $n-t$ is large. 

\subsection*{Proof strategy.}
We combine the spread approximation method developed by Kupavskii and Zakharov~\cite{KZ2024} and refined by Kupavskii~\cite{Kupavskii2026} with uniform estimates for the number of integer partitions into a prescribed number of parts, obtained from Szekeres' formula~\cite{Szekeres1951,Szekeres1953} in the form proved by Canfield~\cite{Canfield1997}. To construct counterexamples for integer partitions with $k$ parts, we derive explicit estimates for ratios of restricted partition numbers from the Szekeres--Canfield formula and apply them to count partitions in families modeled on Frankl's constructions~\cite{Frankl1978}.  To construct counterexamples for unrestricted integer partitions, we use estimates on the partition function $p(n)$ as provided by Nemes~\cite{Nemes2024}.

\subsection*{Organization}
We provide further notation and state Kupavskii's theorem we are going to use in Section~\ref{sec:tools}. In Section~\ref{sec:estimates} we state an elementary estimate of Borg~\cite[Lemma~2.2]{Borg2016} and the Szekeres-Canfield formula~\cite{Canfield1997} we are going to use along with the estimates for partition function. In Section~\ref{sec:upper} we prove Theorem~\ref{thm:main-upper}. Section~\ref{sec:lower} constructs  $t$-intersecting families (of partitions of $n$ into $k$ parts) larger than the canonical $t$-star and proves Theorem~\ref{thm:every-scale}. In Section~\ref{sec:counterexamples_unrestricted} we construct counterexamples to Borg's conjectures on  intersecting families of unrestricted integer partitions. In Section~\ref{sec:conclude} we make further remarks.

\subsection*{Statement on AI use.}
This work arose from applying spread approximation results of Kupavskii~\cite{Kupavskii2026} to integer partitions with a given number of parts. GPT-5.6 Sol was first used to find small counterexamples to Borg's conjecture~1.6~\cite{Borg2016}. The authors then used it to search for larger counterexamples until the general construction was obtained.  Later GPT-6.1 Sol found counterexamples to Borg's conjectures~1.1 and 1.4~\cite{Borg2016}. GPT-5.6 Sol and later GPT-6 Astra and GPT-6.1 Sol also performed the asymptotic calculations based on Szekeres' formula and Nemes's bounds on $p(n)$ and assisted with the exposition. The authors checked all of the proofs carefully.

\section{Occurrence labels and spread tools}\label{sec:tools}

We identify each integer partition with the multiset of its parts.
For a finite multiset $Q$ of positive integers, let $m_Q(a)$ denote the
multiplicity of $a$ in $Q$. We write $Q\subseteq\lambda$ if
\[
 m_Q(a)\le m_\lambda(a)\qquad\text{for every }a\ge1.
\]
Its cardinality and weight are
\[
 |Q|=\sum_{a\ge1}m_Q(a),\qquad
 \sigma(Q)=\sum_{a\ge1}a\,m_Q(a).
\]
The full star with centre $Q$ is
\[
 \mathcal S_Q(n,k)=\{\lambda\in P_{n,k}:Q\subseteq\lambda\}.
\]
Deleting $Q$ gives the exact identity
\begin{equation}\label{eq:multiset-star}
 |\mathcal S_Q(n,k)|=p(n-\sigma(Q),k-|Q|).
\end{equation}
We use $p(0,0)=1$, and set $p(N,L)=0$ when $N<L$, either argument is
negative, or $L=0<N$.

\subsection{Maximum degrees}
Following Borg~\cite[Remark~1.10]{Borg2016}, encode $\lambda\in P_{n,k}$ by the $k$-set
\begin{equation}\label{eq:occurrence}
 S_\lambda=\{(a,j)\in[n]\times[k]:1\le j\le m_\lambda(a)\},\qquad
 \mathcal A_{n,k}=\{S_\lambda:\lambda\in P_{n,k}\}.
\end{equation}
This map is injective and $|S_\lambda\cap S_\mu|=|\lambda\cap\mu|$.
Thus it preserves $t$-intersection and triviality.

Viewing a finite set family $\mathcal A$ as a hypergraph whose vertices
are the elements of the ground set and whose edges are the members of
$\mathcal A$, the degree of an $s$-element set $U$ is the number of edges
containing $U$. We write
\[
 \mathcal A[U]=\{A\in\mathcal A:U\subseteq A\},\qquad
 \Delta_s(\mathcal A)=\max_{|U|=s}|\mathcal A[U]|.
\]
Thus $\Delta_s(\mathcal A)$ is the maximum $s$-degree, or equivalently
the size of a largest full $s$-star.
We also write $\mathcal A[\mathcal S]=\bigcup_{U\in\mathcal S}\mathcal A[U]$.
The following is Borg's largest-star lemma~\cite[Lemma 3.2]{Borg2016}
in maximum-degree notation.

\begin{lemma}[Maximum degrees]\label{lem:degrees}
If $n\ge k$ and $0\le s\le k-1$, then
\begin{equation}\label{eq:degrees}
 \Delta_s(\mathcal A_{n,k})=p(n-s,k-s),
 \qquad \Delta_k(\mathcal A_{n,k})=1.
\end{equation}
\end{lemma}
\begin{proof}
Fix an $s$-set $U$ of occurrence labels and put
$u_a=|\{j:(a,j)\in U\}|$ and $V=\sum_a au_a$.
Every partition whose encoding contains $U$ has at least $u_a$ copies of
$a$.  Deleting these $s$ prescribed parts is an injection into
$P_{n-V,k-s}$.  Since $V\ge s$ and $p(N,k-s)$ is nondecreasing in $N$,
\[
 |\mathcal A_{n,k}[U]|\le p(n-V,k-s)\le p(n-s,k-s).
\]
Equality is attained by $U=\{(1,1),\ldots,(1,s)\}$.
For $s=k$, every $k$-set $U$ satisfies
\[
 |\mathcal A_{n,k}[U]|=
 \begin{cases}
 0,&U\notin\mathcal A_{n,k},\\
 1,&U\in\mathcal A_{n,k},
 \end{cases}
\]
since a $k$-set can contain $U$ only if it equals $U$.
As $n\ge k$, the partition
$\lambda=(\underbrace{1,\ldots,1}_{k-1},n-k+1)$ belongs to $P_{n,k}$,
and choosing $U=S_\lambda$ gives $\Delta_k(\mathcal A_{n,k})=1$.
\end{proof}

\subsection{Weak spreadness}
We use Kupavskii's notion of weak spreadness~\cite{Kupavskii2026},
expressed in terms of maximum degrees as in Wen and Lv~\cite{WenLv2026}.
\begin{definition}[Weakly $(r,t)$-spread families]
A $k$-uniform family $\mathcal A$ is \emph{weakly $(r,t)$-spread} if
\[
 \Delta_{t+h}(\mathcal A)\le r^{-h}\Delta_t(\mathcal A)
 \qquad(0\le h\le k-t).
\]
\end{definition}

The next result converts degree estimates into an intersection bound.
\begin{theorem}[Kupavskii~\cite{Kupavskii2026}]\label{thm:weak}
Let $0<\varepsilon\le1$, and let $\mathcal A$ be a weakly
$(r,t)$-spread $k$-uniform family.  If $\varepsilon r\ge24k$, every
nontrivial $t$-intersecting $\mathcal F\subseteq\mathcal A$ satisfies
\[
 |\mathcal F|\le\varepsilon\Delta_t(\mathcal A).
\]
\end{theorem}
\begin{proof}
Apply \cite[Theorem 14]{Kupavskii2026} with $q=k$ and
$\mathcal S=\mathcal F$.  Our degree condition is its weak-spread
hypothesis, with a maximum $t$-star as reference.  Uniformity gives
$\mathcal A[\mathcal F]=\mathcal F$, so the cited theorem gives the claim.
\end{proof}

\begin{corollary}\label{cor:weak}
Let $\mathcal A$ be a weakly $(r,t)$-spread $k$-uniform family.
If $r\ge24k$, then every $t$-intersecting
$\mathcal F\subseteq\mathcal A$ has size at most $\Delta_t(\mathcal A)$.
If $r>24k$, equality requires a full $t$-star of maximum size.
If $r\ge48k$, every nontrivial such family has size at most
$\frac12\Delta_t(\mathcal A)$.
\end{corollary}
\begin{proof}
A trivial family lies in a full $t$-star.  For a nontrivial family, use
Theorem~\ref{thm:weak} with $\varepsilon=1$ or $1/2$, respectively.
When $r>24k$, the choice $\varepsilon=24k/r<1$ excludes equality for
a nontrivial family.
\end{proof}

To apply the weak-spread bound, we need lower bounds on the ratios of
consecutive maximum degrees. Write $p(m,\le L)$ for the number of
partitions of $m$ into at most $L$ parts. By \eqref{eq:degrees},
\[
 \frac{\Delta_s(\mathcal A_{n,k})}{\Delta_{s+1}(\mathcal A_{n,k})}
 =\frac{p(n-s,k-s)}{p(n-s-1,k-s-1)}
 =\frac{p(m,\le L)}{p(m,\le L-1)},
\]
where $m=n-k$, $L=k-s$, and $0\le s\le k-2$.
The last equality follows by subtracting one from each positive part;
the excess $m$ is the same in the numerator and denominator.
We estimate these ratios using the Szekeres--Canfield formula.
Multiplying the resulting bounds controls the degrees up to $k-1$.
For the final degree $\Delta_k=1$, weak $(r,t)$-spreadness additionally
requires $\Delta_t\ge r^{k-t}$. With $L=k-t$, this is
$p(m,\le L)^{1/L}\ge r$, which we prove separately.

\section{Uniform bounds for partition ratios}\label{sec:estimates}

The conversion used above is the identity
\begin{equation}\label{eq:restricted}
 p(N,L)=p(N-L,\le L).
\end{equation}
We now derive lower bounds for the ratios
$p(m,\le L)/p(m,\le L-1)$, together with bounds of the form
$p(m,\le L)\ge r^L$ needed to control the final degree.
We first give an elementary bound due to Borg, then use the
Szekeres--Canfield formula to obtain sharper estimates.

\subsection{An elementary ratio bound}
Borg's estimate handles the ranges in which only a coarse bound is needed.
\begin{lemma}[Borg~\cite{Borg2016}]\label{lem:borg}
If $k\ge2$, $c\ge1$ are integers and $n\ge ck^3$, then
\[
 p(n,k)>c\,p(n,k-1)\ge c\,p(n-1,k-1).
\]
Consequently, if $q\ge2$ and $m+q\ge q^3$, then
\begin{equation}\label{eq:borg-ratio}
 \frac{p(m,\le q)}{p(m,\le q-1)}
 >\left\lfloor\frac{m+q}{q^3}\right\rfloor
 \ge\frac{m+q}{q^3}-1.
\end{equation}
\end{lemma}
\begin{proof}
The first assertion is \cite[Lemma 2.2]{Borg2016}.  Apply it with
$n=m+q$, $k=q$ and $c=\lfloor(m+q)/q^3\rfloor$, and then use
\eqref{eq:restricted}.
\end{proof}

\subsection{Explicit identities and partition estimates}
To estimate the ratios $p(m,\le L)/p(m,\le L-1)$ more precisely, we use
Szekeres' formula~\cite{Szekeres1951,Szekeres1953} in the uniform form
proved by Canfield~\cite{Canfield1997}. We first introduce the notation
used in the formula and in the estimates below.

For $u>0$, let $v=v(u)>0$ be determined by
\begin{equation}\label{eq:saddle}
 I(v)=\int_0^v\frac{x}{e^x-1}\,dx,
 \qquad u^2=\frac{v^2}{I(v)}.
\end{equation}
The solution $v=v(u)>0$ is unique. Define
\begin{subequations}\label{eq:fg}
\begin{align}
 g(u)&=\frac{2v}{u}-u\log(1-e^{-v}),
       \label{eq:g-definition}\\
 f(u)&=\frac{v}{2^{3/2}\pi u}
       \left(1-e^{-v}-\frac{u^2e^{-v}}2\right)^{-1/2}.
       \label{eq:f-definition}
\end{align}
\end{subequations}
We also write
\begin{subequations}\label{eq:q-Phi-definitions}
\begin{align}
 q(u)&=1-e^{-v(u)},\label{eq:q-definition}\\
 \Phi(M,L)&=\sqrt M\,g\!\left(\frac L{\sqrt M}\right)
 \quad(M,L>0).\label{eq:Phi-definition}
\end{align}
\end{subequations}
At a point $(M,L)$, we use the scaled variable $u=L/\sqrt M$.
Subscripts denote partial derivatives: $\Phi_L=\partial\Phi/\partial L$
with $M$ held fixed, and $\Phi_M=\partial\Phi/\partial M$ with $L$
held fixed. Second partial derivatives are denoted by $\Phi_{LL}$,
$\Phi_{ML}$, and $\Phi_{MM}$.

\begin{theorem}[Canfield~\cite{Canfield1997}]\label{thm:Szekeres-Canfield}
For every fixed $\delta>0$, as $m\to\infty$ through positive integers,
uniformly over integers $L\ge1$,
\begin{equation}\label{eq:canfield}
 p(m,\le L)=\frac{f(u)}m
 \exp\!\left(\sqrt m\,g(u)
       +O(m^{-1/6+\delta}+L^{-1})\right),\qquad u=\frac L{\sqrt m},
\end{equation}
where the implied constant may depend on $\delta$.
The $L^{-1}$ term can be omitted when $L\ge m^{1/6}$.
\end{theorem}

Thus $\Phi$ is the exponent in the Szekeres--Canfield formula.
The next lemma gives explicit derivative identities and bounds for
this exponent and the prefactor.
\begin{lemma}\label{lem:smallu}
The function $v$ is continuously differentiable and $g$ is twice
continuously differentiable on $(0,\infty)$. The function $\Phi$ is
twice continuously differentiable for $M,L>0$, and
\begin{equation}\label{eq:saddle-derivatives}
 \Phi_L(M,L)=g'(u)=-\log q(u),\qquad \Phi_M(M,L)=\frac{v(u)}L,
 \qquad u=\frac L{\sqrt M}.
\end{equation}
For all $u>0$, the following bounds hold:
\begin{subequations}\label{eq:global-function-bounds}
\begin{align}
 0<v(u)&\le u^2,\label{eq:v-global-bound}\\
 0<q(u)&\le v(u)\le u^2,\label{eq:q-global-bound}\\
 0<v'(u)&\le\frac{2v(u)}u\le2u,\label{eq:v-prime-global-bound}\\
 |g''(u)|&\le\frac2u.\label{eq:g-second-global-bound}
\end{align}
\end{subequations}
The function $f$ has a positive continuous extension to $[0,\infty)$, with
\begin{equation}\label{eq:f-positive-extension}
 f(0)=\frac1{2\pi},\qquad f(u)>0\quad(u\ge0).
\end{equation}
As $u\to0$,
\begin{equation}\label{eq:small-u-asymptotics}
 v(u)=u^2(1+O(u^2)),\qquad q(u)=u^2(1+O(u^2)),
 \qquad f(u)=\frac{1+O(u^2)}{2\pi}.
\end{equation}
For every fixed $U>0$, there is a finite constant $C_U$, depending only
on $U$, such that
\begin{equation}\label{eq:log-f-uniform-bound}
 |(\log f)'(u)|\le C_Uu\qquad(0<u\le U).
\end{equation}
Equivalently, $(\log f)'(u)=O_U(u)$, where the implied constant is
independent of $u$.
\end{lemma}
\begin{proof}
\emph{The implicit parameter and its derivative.}
Write
\begin{equation*}\tag{P1}\label{eq:rho-definition}
 \rho(v)=\frac{v}{e^v-1}=I'(v),
\end{equation*}
and define
\begin{equation*}\tag{P2}\label{eq:H-definition}
 H(v)=2I(v)-v\rho(v).
\end{equation*}
The function $\rho$ is positive and strictly decreasing, with
$\rho(0)=1$ by continuous extension. Indeed, differentiating
\eqref{eq:rho-definition} gives
$\rho'(v)=(e^v(1-v)-1)/(e^v-1)^2<0$ for $v>0$.
Integrating the inequalities $\rho(v)\le\rho(x)\le1$ over
$0\le x\le v$ and using $I'=\rho$ from \eqref{eq:rho-definition} gives
\begin{equation*}\tag{P3}\label{eq:I-bounds}
 v\rho(v)\le I(v)\le v.
\end{equation*}
Using $v\rho(v)\le I(v)$ from \eqref{eq:I-bounds} in the definition
$H(v)=2I(v)-v\rho(v)$ from \eqref{eq:H-definition}, and noting that
$I(v)>0$ for $v>0$, we obtain
\begin{equation*}\tag{P4}\label{eq:H-lower-bound}
 H(v)\ge I(v)>0.
\end{equation*}
By the relation $u^2=v^2/I(v)$ from \eqref{eq:saddle} and the upper bound $I(v)\le v$
in \eqref{eq:I-bounds}, we obtain
\begin{equation*}\tag{P5}\label{eq:v-bound}
 0<v(u)\le u^2.
\end{equation*}
Since $q(u)=1-e^{-v(u)}$ by \eqref{eq:q-definition}
and $1-e^{-v}\le v$ for $v>0$, this also gives
\begin{equation*}
 0<q(u)=1-e^{-v(u)}\le v(u)\le u^2.
\end{equation*}
In particular, \eqref{eq:v-bound} proves $v(u)=O(u^2)$ with constant one.

Put $F(v)=v^2/I(v)$. This function is continuously differentiable on
$(0,\infty)$. Using $I'(v)=\rho(v)$ from \eqref{eq:rho-definition}
and $2I(v)-v\rho(v)=H(v)$ from \eqref{eq:H-definition}, we find
\[
 F'(v)=\frac{vH(v)}{I(v)^2}>0.
\]
The strict positivity follows from $H(v)\ge I(v)>0$ in \eqref{eq:H-lower-bound}.
The inverse function theorem therefore implies that
$v(u)=F^{-1}(u^2)$ is continuously differentiable for $u>0$.
We may consequently differentiate the implicit equation with respect to
$u$, obtaining
\[
 2u=\frac{vH(v)}{I(v)^2}\,v'(u),
\]
and hence
\begin{equation*}\tag{P6}\label{eq:v-prime}
 v'(u)=\frac{2uI(v)^2}{vH(v)}
       =\frac{2vI(v)}{uH(v)}.
\end{equation*}
The last equality uses $u^2I(v)=v^2$ from \eqref{eq:saddle}.
Inserting $H(v)\ge I(v)>0$ from \eqref{eq:H-lower-bound} into the formula
$v'(u)=2vI(v)/(uH(v))$ in \eqref{eq:v-prime}, and then using
$v(u)\le u^2$ from \eqref{eq:v-bound}, gives
\begin{equation*}\tag{P7}\label{eq:v-prime-bound}
 0<v'(u)\le\frac{2v(u)}u\le2u,
\end{equation*}
which proves $v'(u)=O(u)$ on all of $(0,\infty)$.

\emph{The identities for $g$ and the partial derivatives of $\Phi$.}
By \eqref{eq:rho-definition}, $q'/q=v'/(e^v-1)=\rho(v)v'/v$.
Using this identity, differentiate the definition \eqref{eq:g-definition}, written as
$g(u)=2v/u-u\log q(u)$:
\begin{align*}
 g'(u)
 &=\left(\frac2u-\frac{u\rho(v)}v\right)v'
     -\frac{2v}{u^2}-\log q(u)\notag\\
 &=\frac{H(v)}{uI(v)}v'-\frac{2v}{u^2}-\log q(u)
  =-\log q(u).\tag{P8}\label{eq:g-prime}
\end{align*}
In the second line, $u^2=v^2/I(v)$ from \eqref{eq:saddle} and
$H(v)=2I(v)-v\rho(v)$ from \eqref{eq:H-definition} turn the coefficient
of $v'$ into $H(v)/(uI(v))$. Substituting
$v'=2vI(v)/(uH(v))$ from \eqref{eq:v-prime} then cancels the two terms $2v/u^2$.

For $u=L/\sqrt M$, as in the definition of $\Phi$ in \eqref{eq:Phi-definition},
the chain rule gives
\[
 \frac{\partial u}{\partial L}=\frac1{\sqrt M},\qquad
 \frac{\partial u}{\partial M}=-\frac{u}{2M}.
\]
Using $g'(u)=-\log q(u)$ from \eqref{eq:g-prime}
and holding the other variable fixed in each case,
\begin{align*}
 \Phi_L(M,L)
 &=\sqrt M\,g'(u)\frac1{\sqrt M}
   =-\log q(u),\\
 \Phi_M(M,L)
 &=\frac{g(u)}{2\sqrt M}
     -\frac{u g'(u)}{2\sqrt M}
   =\frac{2v/u}{2\sqrt M}=\frac{v(u)}L.
\end{align*}
Here $g(u)=2v/u-u\log q(u)$ by \eqref{eq:g-definition}, while
$ug'(u)=-u\log q(u)$ by \eqref{eq:g-prime}. The logarithmic terms
therefore cancel in $g(u)-ug'(u)$, leaving $2v/u$. This proves
\eqref{eq:saddle-derivatives}.
Differentiating $g'(u)=-\log q(u)$ from \eqref{eq:g-prime} once more gives
\begin{equation*}\tag{P9}\label{eq:g-second}
 g''(u)=-\frac{q'(u)}{q(u)}
       =-\frac{v'(u)}{e^{v(u)}-1}.
\end{equation*}
Using $g''(u)=-v'(u)/(e^{v(u)}-1)$ from \eqref{eq:g-second},
the inequality $e^v-1\ge v$, and the bound
$v'(u)\le2v(u)/u$ from \eqref{eq:v-prime-bound}, we obtain
\begin{equation*}
 |g''(u)|\le\frac{v'(u)}{v(u)}\le\frac2u.
\end{equation*}
This proves $g''(u)=O(u^{-1})$ with constant two. Since $v$ is
continuously differentiable and $e^{v(u)}-1>0$ for $u>0$,
\eqref{eq:g-second} also shows that $g''$ is continuous on $(0,\infty)$.
Thus $g$ is twice continuously differentiable. By
\eqref{eq:saddle-derivatives}, both first partial derivatives of $\Phi$
are continuously differentiable for $M,L>0$. Hence $\Phi$ has continuous second partial
derivatives, as required for the Taylor formula used later.

\emph{Behaviour at zero and the prefactor bound.}
Taylor's formula gives $e^v-1=v(1+v/2+O(v^2))$.
Substituting this into $\rho(v)=v/(e^v-1)$ from \eqref{eq:rho-definition}
and using the reciprocal expansion $(1+z)^{-1}=1-z+O(z^2)$ as $z\to0$, we obtain
\[
 \rho(v)=\frac{1}{1+v/2+O(v^2)}=1-v/2+O(v^2).
\]
Integration gives $I(v)=v-v^2/4+O(v^3)$.
Using the reciprocal expansion once more in $u^2=v^2/I(v)$ from
\eqref{eq:saddle}, we find
\[
 u^2=\frac{v}{1-v/4+O(v^2)}
     =v\bigl(1+v/4+O(v^2)\bigr)=v+O(v^2).
\]
Since $v(u)\le u^2$ by \eqref{eq:v-bound}, the last error is $O(u^4)$, and therefore
\[
 v(u)=u^2+O(u^4).
\]
Since $v(u)\to0$ as $u\to0$, Taylor expansion of $e^{-v(u)}$ gives
\begin{align*}
 q(u)=1-e^{-v(u)}
 &=v(u)-\frac{v(u)^2}{2}+O(v(u)^3)\\
 &=v(u)+O(v(u)^2)=u^2+O(u^4).
\end{align*}

To handle the apparent singularities in $f$, introduce functions of
$v$ alone:
\[
 J(v)=\frac{I(v)}v,\qquad K(v)=\frac{1-e^{-v}}v.
\]
These functions, together with $\rho$, have continuously differentiable
extensions to zero. This follows directly from the integral identities
\[
 \rho(v)=\left(\int_0^1 e^{sv}\,ds\right)^{-1},\qquad
 J(v)=\int_0^1\rho(sv)\,ds,\qquad
 K(v)=\int_0^1e^{-sv}\,ds.
\]
Their values at zero are
\begin{equation*}\tag{P10}\label{eq:auxiliary-values-zero}
 \rho(0)=J(0)=K(0)=1.
\end{equation*}
By the definitions of $J$ and $K$, the definition of $\rho$ in
\eqref{eq:rho-definition}, and the relation $u^2=v^2/I(v)$ in
\eqref{eq:saddle}, we have
\[
 1-e^{-v}=vK(v),\qquad
 u^2=\frac{v}{J(v)},\qquad
 e^{-v}=K(v)\rho(v).
\]
For the last identity, note that
\[
 K(v)\rho(v)=\frac{1-e^{-v}}{e^v-1}=e^{-v}.
\]
Substitution and factoring out $vK(v)$ now give
\begin{align*}
 1-e^{-v}-\frac{u^2e^{-v}}2
 &=vK(v)-\frac{vK(v)\rho(v)}{2J(v)}\\
 &=\frac{vK(v)(2J(v)-\rho(v))}{2J(v)}.
\end{align*}
Also, $u^2=v/J(v)$ gives $v/u=\sqrt{vJ(v)}$.
Substituting these expressions into the formula for $f$ in
\eqref{eq:f-definition} cancels the factors $\sqrt v$
and gives the nonsingular expression
\begin{equation*}\tag{P11}\label{eq:f-regular}
 f(u)=\frac1{2\pi}\,
       \frac{J(v(u))}{\sqrt{K(v(u))(2J(v(u))-\rho(v(u)))}}.
\end{equation*}
All factors under the square root are positive: $K(v)>0$ and
$2J(v)-\rho(v)=H(v)/v\ge J(v)>0$, using \eqref{eq:H-definition}
for the equality and \eqref{eq:H-lower-bound} for the inequalities.
At zero, \eqref{eq:auxiliary-values-zero} gives $K(0)=1$ and
$2J(0)-\rho(0)=1$, so both factors remain positive.
Thus the expression on the right, regarded as a function of $v\ge0$,
is positive and continuously differentiable. Substituting
$\rho(0)=J(0)=K(0)=1$ from \eqref{eq:auxiliary-values-zero} into
\eqref{eq:f-regular} gives its value $1/(2\pi)$ at zero.
It is therefore $1/(2\pi)+O(v)$ near zero. Since $v(u)\le u^2$ by \eqref{eq:v-bound},
this proves the expansion for $f$ in \eqref{eq:small-u-asymptotics}.
Together with the continuity of $v(u)$ for $u>0$ and $v(u)\to0$ as
$u\to0$, the positive continuous expression \eqref{eq:f-regular}
also gives the extension asserted in \eqref{eq:f-positive-extension}.

Finally, logarithmic differentiation of \eqref{eq:f-regular} gives
\[
 (\log f)'(u)=D(v(u))v'(u),\qquad
 D(v)=\frac{J'(v)}{J(v)}-\frac{K'(v)}{2K(v)}
       -\frac{2J'(v)-\rho'(v)}{2(2J(v)-\rho(v))}.
\]
Here the primes in $D$ denote differentiation with respect to $v$.
Continuous differentiability at zero and positivity of the denominators show
that $D$ is continuous on $[0,U^2]$. Thus
$B_U:=\max_{0\le v\le U^2}|D(v)|<\infty$.
By \eqref{eq:v-bound}, $v(u)\le u^2\le U^2$ for $0<u\le U$.
Using $v'(u)\le2u$ from \eqref{eq:v-prime-bound}, we obtain
\begin{equation*}
 |(\log f)'(u)|\le 2B_Uu\qquad(0<u\le U),
\end{equation*}
which proves \eqref{eq:log-f-uniform-bound} with $C_U=2B_U$.
This proves every assertion.
\end{proof}

We now use the derivative bounds from Lemma~\ref{lem:smallu} to compare
the Szekeres--Canfield formula at consecutive values of $L$.
Together with Borg's bound from Lemma~\ref{lem:borg}, this yields the
estimates needed to control consecutive maximum degrees and the final degree.

\begin{lemma}[Consecutive degree ratios and the final degree]\label{lem:adjacent}
Fix $a>0$ and $0<\varepsilon<a$. For all sufficiently large $k$,
all integers $1\le s\le k$, all $m\ge ask^2$, and all $2\le L\le k$
satisfy
\begin{equation}\label{eq:adjacent}
 \frac{p(m,\le L)}{p(m,\le L-1)}\ge(a-\varepsilon)s,
 \qquad p(m,\le L)^{1/L}\ge(a-\varepsilon)s.
\end{equation}
\end{lemma}
\begin{proof}
First suppose $m\le k^5$ and $L\ge k^{1/3}$. The full formula
\eqref{eq:canfield}, with a fixed $0<\delta<1/6$, has a uniform
$o(1)$ error at $L$ and $L-1$: indeed $m\ge ak^2$ and
$L-1\ge\frac12k^{1/3}$ for sufficiently large $k$.
Write its logarithmic form as
\begin{equation*}\tag{P12}\label{eq:log-partition-remainder}
 \log p(m,\le j)=\Phi(m,j)+\log f(j/\sqrt m)-\log m+R_{m,j},
 \qquad \max_{j\in\{L-1,L\}}|R_{m,j}|=o(1).
\end{equation*}
All errors below are uniform over the current parameter range.

\emph{The ratio estimate.}
Let $x\in[L-1,L]$ and put $u_x=x/\sqrt m$.
Since $x\le k$ and $m\ge ask^2\ge ak^2$,
we have $0<u_x\le a^{-1/2}$. The identity $\Phi_L(m,x)=-\log q(u_x)$
from \eqref{eq:saddle-derivatives} and the bound $q(u_x)\le u_x^2$
from \eqref{eq:q-global-bound} give
\[
 \Phi_L(m,x)=-\log q(u_x)
 \ge-\log(u_x^2)=\log\frac m{x^2}
 \ge\log\frac m{L^2}.
\]
The last inequality uses $x\le L$. The fundamental theorem of calculus
therefore yields
\[
 \Phi(m,L)-\Phi(m,L-1)
 =\int_{L-1}^{L}\Phi_L(m,x)\,dx
 \ge\log\frac m{L^2}.
\]
For the prefactor we use $|(\log f)'(u)|\le C_Uu$ from
\eqref{eq:log-f-uniform-bound}, with $U=a^{-1/2}$ and a suitable constant
$C_a$ depending only on $a$:
\begin{align*}
 \left|\log f(L/\sqrt m)-\log f((L-1)/\sqrt m)\right|
 &\le C_a\int_{(L-1)/\sqrt m}^{L/\sqrt m} y\,dy\\
 &=C_a\frac{2L-1}{2m}\le C_a\frac Lm=o(1).
\end{align*}
Here $L/m\le1/(ak)$. Subtracting the two logarithmic partition
estimates in \eqref{eq:log-partition-remainder} cancels $-\log m$, and hence
\[
 \log\frac{p(m,\le L)}{p(m,\le L-1)}
 \ge\log\frac m{L^2}-o(1).
\]
Exponentiating gives
\[
 \frac{p(m,\le L)}{p(m,\le L-1)}
 \ge e^{-o(1)}\frac m{L^2}\ge e^{-o(1)}as,
\]
where the last inequality uses $m\ge ask^2$ and $L\le k$.

\emph{The root estimate.}
With $u=L/\sqrt m$, use $g(u)=2v/u-u\log q(u)$ from
\eqref{eq:g-definition}, $u^2=v^2/I(v)$ from \eqref{eq:saddle},
and $q(u)\le u^2$ from \eqref{eq:q-global-bound} to obtain
\begin{equation*}\tag{P13}\label{eq:root-main-term}
 \frac{g(u)}u
 =\frac{2v}{u^2}-\log q(u)
 =\frac{2I(v)}v-\log q(u)
 \ge-\log q(u)\ge\log\frac m{L^2}.
\end{equation*}
We used $I(v)>0$, which follows from its integral definition in
\eqref{eq:saddle}, to discard the first term. By Lemma~\ref{lem:smallu},
$f$ extends continuously to $[0,a^{-1/2}]$ and is positive there,
including $f(0)=1/(2\pi)$ by \eqref{eq:f-positive-extension}.
By enlarging $C_a$ if necessary, we have
$|\log f(u)|\le C_a$ for all $u\in[0,a^{-1/2}]$.
Since $m\le k^5$ and $L\ge k^{1/3}$,
\begin{equation*}\tag{P14}\label{eq:root-prefactor-error}
 \left|\frac{\log f(u)-\log m}{L}\right|
 \le\frac{C_a+5\log k}{k^{1/3}}=o(1).
\end{equation*}
Dividing \eqref{eq:log-partition-remainder} at $j=L$ by $L$ and using
$\Phi(m,L)/L=g(u)/u$, we apply the lower bound for $g(u)/u$ in
\eqref{eq:root-main-term} and the prefactor error estimate in
\eqref{eq:root-prefactor-error}. Together with $R_{m,L}/L=o(1)$, this gives
\[
 \frac1L\log p(m,\le L)\ge\log\frac m{L^2}-o(1).
\]
Exponentiating this last bound gives
$p(m,\le L)^{1/L}\ge e^{-o(1)}m/L^2\ge e^{-o(1)}as$.
Together with the ratio estimate above, this bounds both quantities in
\eqref{eq:adjacent} from below. The errors are uniform, so for
sufficiently large $k$ this is at least $(a-\varepsilon)s$.

If $L<k^{1/3}$ or $m>k^5$, Borg's bound \eqref{eq:borg-ratio} gives
a ratio greater than $(m+L)/L^3-1$. After division by $s$, this
lower bound is at least $ak-1$ when $L<k^{1/3}$, using $m\ge ask^2$,
and at least $k-1$ when $m>k^5$, using $L,s\le k$.
It therefore tends to infinity uniformly as $k\to\infty$.
For the root estimate, there are $\binom{m-1}{L-1}$ ordered
representations of $m$ as a sum of $L$ positive integers.
Since each partition into exactly $L$ parts gives rise to at most $L!$
such representations, we have
\[
 p(m,\le L)\ge p(m,L)\ge\frac1{L!}\binom{m-1}{L-1}.
\]
Using $\binom Nj\ge(N/j)^j$ and $L!\le L^L$, we obtain, for $L\ge3$,
\[
 p(m,\le L)^{1/L}
 \ge\frac{(m-1)^{(L-1)/L}}{L(L-1)^{(L-1)/L}}
 \ge\frac{(m-1)^{2/3}}{L^2}.
\]
The last inequality uses $(L-1)/L\ge2/3$ and
$(L-1)^{(L-1)/L}\le L$.
This is $\Omega_a(s^{2/3}k^{2/3})$ when $L<k^{1/3}$,
and $\Omega(k^{4/3})$ when $m>k^5$.
Finally, for $L=2$, we have $p(m,\le2)=\lfloor m/2\rfloor+1$, whose square root is
$\Omega_a(k\sqrt s)$. Since $s\le k$, after division by $s$ these
lower bounds are at least a positive constant depending only on $a$
times $k^{1/3}$ for $L\ge3$, and $k^{1/2}$ for $L=2$.
They therefore tend to infinity uniformly as $k\to\infty$.
Thus both quantities in \eqref{eq:adjacent} are at least
$(a-\varepsilon)s$ for all sufficiently large $k$, completing the proof.
\end{proof}

\section{Proof of Theorem~\ref{thm:main-upper}}\label{sec:upper}

\begin{proof}[Proof of Theorem~\ref{thm:main-upper}]
If $t=k-1$, two partitions sharing $k-1$ parts must coincide: their
remaining parts agree because the total sums agree.  Thus every such
family has size at most one, and every nonempty one is trivial.
Every singleton is a full $(k-1)$-star, so the equality statement also holds.

Suppose $L=k-t\ge2$ and put $m=n-k$.  Lemma~\ref{lem:degrees} gives
\[
 \Delta_{t+h}(\mathcal A_{n,k})=p(m,\le L-h)
       \quad(0\le h\le L-1),\qquad
 \Delta_k(\mathcal A_{n,k})=1.
\]
Choose constants $24<c<C<A$.  For all sufficiently large $k$,
$m\ge Ck^3$. Apply Lemma~\ref{lem:adjacent} with
$a=C$, $\varepsilon=C-c$, and $s=k$. Writing $r=ck$, we obtain
\[
 \frac{p(m,\le q)}{p(m,\le q-1)}\ge r\quad(2\le q\le L),
 \qquad p(m,\le L)\ge r^L.
\]
For $1\le h\le L-1$, multiplying the ratio bounds gives
\[
 \frac{\Delta_t(\mathcal A_{n,k})}
      {\Delta_{t+h}(\mathcal A_{n,k})}
 =\frac{p(m,\le L)}{p(m,\le L-h)}
 =\prod_{q=L-h+1}^{L}\frac{p(m,\le q)}{p(m,\le q-1)}
 \ge r^h.
\]
For $h=L$, the root bound gives
$\Delta_t(\mathcal A_{n,k})=p(m,\le L)\ge r^L\Delta_k(\mathcal A_{n,k})$.
Thus $\Delta_{t+h}(\mathcal A_{n,k})\le r^{-h}\Delta_t(\mathcal A_{n,k})$
for every $0\le h\le L$, with equality when $h=0$.
Hence $\cA_{n,k}$ is weakly $(r,t)$-spread.  Since $r>24k$,
Corollary~\ref{cor:weak} proves the size bound and shows that
equality holds precisely for full $t$-stars of maximum size.
If $k\ge t+3$, Borg's largest-star lemma
\cite[Lemma~3.2]{Borg2016} shows that the canonical star is
the unique such family.
If $A>48$, choose $48<c<C<A$ in the same argument; then $r>48k$
and the corollary gives the half-star bound.
\end{proof}

\begin{remark}
Using Theorem~1.3(i) of the preprint by Wen and Lv~\cite{WenLv2026},
the size bound and the equality characterization in
Theorem~\ref{thm:main-upper} extend to every fixed $A>2e$.
Indeed, choosing $2e<c<C<A$, the same partition estimates yield
weak $(r,t)$-spreadness with $r=ck>2ek$, allowing us to apply their theorem.
\end{remark}

\section{Proof of Theorem~\ref{thm:every-scale}}
\label{sec:lower}

The following construction is the first nontrivial complete-intersection pattern:
require all but at most one of a small collection of distinct part values.  This is  the $i=1$ pattern 
for $t$-intersecting set families mentioned in the introduction~\cite{Frankl1978}.

For $1\le t\le k-2$, put
\[
 C=[t+2],\qquad S=\frac{(t+2)(t+3)}2,
 \qquad
 \mathcal F_{n,k,t}
 =\{\lambda\in P_{n,k}:|\operatorname{supp}(\lambda)\cap C|
 \ge t+1\}.
\]

\begin{lemma}[The construction and its size]
\label{lem:construction-size}
The family $\mathcal F_{n,k,t}$ is $t$-intersecting, and
\begin{equation}
\label{eq:construction-size}
 |\mathcal F_{n,k,t}|
 =\sum_{i=1}^{t+2}p(n-S+i,k-t-1)
 -(t+1)p(n-S,k-t-2).
\end{equation}
\end{lemma}
\begin{proof}
Two subsets of a $(t+2)$-set, each containing at least $t+1$ elements,
have at least $t$ common elements. Thus any two members have at least
$t$ common distinct part values. To count the family, sum over the
possible omitted value $i\in C$. Deleting the required copy of every
value in $C\setminus\{i\}$ gives $p(n-S+i,k-t-1)$ partitions.
Members containing all of $C$ have been counted $t+2$ times; deleting one
copy of every value in $C$ counts them as $p(n-S,k-t-2)$.
\end{proof}

For example,
\[
 |\mathcal F_{31,8,1}|
 =p(28,6)+p(27,6)+p(26,6)-2p(25,5)
 =391+331+282-2\cdot192=620,
\]
whereas the canonical star has $p(30,7)=618$ members.

To compare our construction with the canonical star, we divide the size
formula \eqref{eq:construction-size} by $p(n-t,k-t)$. In
Lemma~\ref{lem:uniform-size-ratio}, we estimate the resulting ratios of
partition numbers using the Szekeres--Canfield formula and the derivative
bounds from Lemma~\ref{lem:smallu}, now allowing both arguments to vary.

\subsection{A uniform comparison}

Write
\[
 I(v)=\int_0^v\frac{x}{e^x-1}\,dx.
\]
The function $I(v)/v^2$ decreases continuously from infinity to zero,
so the parameter $v$ in the next lemma is uniquely determined.

\begin{lemma}[Uniform size ratio]
\label{lem:uniform-size-ratio}
Fix $0<c_0<c_1$ and $\eta>0$. Uniformly for
$c\in[c_0,c_1]$, $n=\lfloor ctk^2\rfloor$,
$1\le t\le k-2$, and $d=k-t\ge\eta k$, put
\[
 m=n-k,\qquad E_0=\frac{(t+1)(t+2)}2,
 \qquad \frac{m}{d^2}=\frac{I(v)}{v^2},\qquad q=1-e^{-v}.
\]
Then, as $k\to\infty$,
\begin{equation}
\label{eq:uniform-size-ratio}
 \frac{|\mathcal F_{n,k,t}|}{p(n-t,k-t)}
 =(1+o(1))e^{-vE_0/d}
 \bigl((t+2)q-(t+1)q^2\bigr).
\end{equation}
\end{lemma}

\begin{proof}
Use $p(N,L)=p(N-L,\le L)$. The denominator is $p(m,\le d)$,
and the positive
terms in \eqref{eq:construction-size} are
$p(m-E_i,\le d-1)$ with $E_i=E_0+1-i$, and the subtracted term is
$p(m-E_0,\le d-2)$.

We compare $p(m-E,\le d-a)$ with $p(m,\le d)$ for $a\in\{1,2\}$ and
integers $0\le E\le E_0$; this covers all terms above.
All implicit constants below may depend on $c_0,c_1,\eta$, but not on
$c,t,k,E,a$ or $i$.

\emph{The rectangles and the range of Lemma~\ref{lem:smallu}.}
For sufficiently large $k$, uniformly in the parameters,
\[
 \frac{c_0}{2}tk^2\le m\le c_1tk^2,\qquad
 \eta k\le d\le k,\qquad E_0\le3t^2.
\]
In particular, $E_0/m=O(t/k^2)=O(k^{-1})$ and $2/d=O(k^{-1})$.
For each displacement $(-E,-a)$, define the closed rectangle
\[
 \mathcal R_{E,a}=[m-E,m]\times[d-a,d].
\]
This is the ``Taylor rectangle'': it contains both endpoints and the
straight segment joining them. All these rectangles lie in the common
rectangle
\[
 \mathcal R=[m-E_0,m]\times[d-2,d].
\]
At every real point $(M,L)\in\mathcal R$, for sufficiently large $k$,
\[
 \frac m2\le M\le m,\qquad \frac d2\le L\le d,\qquad
 0<u:=\frac L{\sqrt M}\le\frac{2}{\sqrt{c_0t}}
 \le U:=\frac2{\sqrt{c_0}}.
\]
Thus the same fixed $U$ in Lemma~\ref{lem:smallu} works throughout
every rectangle. The variables $M,L$ may vary continuously here;
the partition function itself will only be evaluated at integer endpoints.

\emph{First and second derivatives of the exponent.}
Recall $\Phi(M,L)=\sqrt M\,g(L/\sqrt M)$ from \eqref{eq:Phi-definition}.
The identities for $\Phi_L=g'$ and $\Phi_M$ in
\eqref{eq:saddle-derivatives} give, at a variable point of the rectangle,
\[
 \Phi_L(M,L)=g'(u)=-\log(1-e^{-v(u)}),\qquad
 \Phi_M(M,L)=\frac{v(u)}L.
\]
The symbols $v$ and $q$ without arguments in the statement of the lemma
refer to the base point: $v=v(d/\sqrt m)$ and $q=1-e^{-v}$.
Thus $\Phi_L(m,d)=-\log q$ and $\Phi_M(m,d)=v/d$.

Using $u_L=1/\sqrt M$ and $u_M=-u/(2M)$, differentiate once more:
\begin{align*}
 \Phi_{LL}(M,L)&=\frac{g''(u)}{\sqrt M},\\
 \Phi_{ML}(M,L)&=-\frac{u g''(u)}{2M},\\
 \Phi_{MM}(M,L)&=-\frac{u v'(u)}{2ML}.
\end{align*}
These subscripts denote the indicated second partial derivatives.
The bounds $|g''(u)|\le2/u$ from \eqref{eq:g-second-global-bound}
and $0<v'(u)\le2u$ from \eqref{eq:v-prime-global-bound}
imply the explicit estimates
\[
 |\Phi_{LL}(M,L)|\le\frac2L,\qquad
 |\Phi_{ML}(M,L)|\le\frac1M,\qquad
 |\Phi_{MM}(M,L)|\le\frac{L}{M^2}.
\]
For the last estimate we used $u^2=L^2/M$.
Since $M\ge m/2$ and $d/2\le L\le d$, this gives uniformly on
$\mathcal R$
\[
 |\Phi_{LL}|\le\frac4d,\qquad
 |\Phi_{ML}|\le\frac2m,\qquad
 |\Phi_{MM}|\le\frac{4d}{m^2}.
\]

\emph{The two-variable Taylor formula.}
To make the remainder explicit, restrict $\Phi$ to the segment by putting
\[
 \psi(s)=\Phi(m-sE,d-sa),\qquad 0\le s\le1.
\]
The chain rule gives
\begin{align*}
 \psi'(0)&=-E\Phi_M(m,d)-a\Phi_L(m,d)
             =-\frac{Ev}{d}+a\log q,\\
 \psi''(s)&=E^2\Phi_{MM}(m-sE,d-sa)
       +2Ea\Phi_{ML}(m-sE,d-sa)
       +a^2\Phi_{LL}(m-sE,d-sa).
\end{align*}
The mixed term has coefficient $2Ea$ because both coordinates change:
$2(-E)(-a)=2Ea$. Ordinary Taylor's formula with integral remainder,
\[
 \psi(1)=\psi(0)+\psi'(0)
 +\underbrace{\int_0^1(1-s)\psi''(s)\,ds}_{R_{E,a}},
\]
is precisely the first-order two-variable Taylor formula along this
displacement. It yields
\[
 \Phi(m-E,d-a)-\Phi(m,d)
 =a\log q-\frac{Ev}{d}+R_{E,a},
\]
where, since the whole segment lies in $\mathcal R$,
\begin{align*}
 |R_{E,a}|
 &\le\frac12\left(
 E^2\sup_{\mathcal R}|\Phi_{MM}|
 +2Ea\sup_{\mathcal R}|\Phi_{ML}|
 +a^2\sup_{\mathcal R}|\Phi_{LL}|\right)\\
 &\le\frac{2E^2d}{m^2}+\frac{2Ea}{m}+\frac{2a^2}{d}.
\end{align*}
Here the factor $1/2$ is $\int_0^1(1-s)\,ds$.
Each term is uniformly $O(k^{-1})$: indeed, $a\le2$, $E\le3t^2$,
$d\asymp k$, $m\asymp tk^2$, and $t\le k$ give
\[
 \frac{a^2}{d}=O(k^{-1}),\qquad
 \frac{Ea}{m}=O(t/k^2)=O(k^{-1}),\qquad
 \frac{E^2d}{m^2}=O(t^2/k^3)=O(k^{-1}).
\]

\emph{The prefactor and the partition asymptotic.}
The logarithm of the Szekeres--Canfield prefactor is
\[
 B(M,L)=\log f(L/\sqrt M)-\log M.
\]
Writing $h(u)=\log f(u)$, the estimate \eqref{eq:log-f-uniform-bound}
gives $h'(u)=O_U(u)$ and hence
\[
 B_L=\frac{h'(u)}{\sqrt M}=O_U(L/M),\qquad
 B_M=-\frac{u h'(u)}{2M}-\frac1M=O_U(M^{-1}).
\]
For clarity, if $|h'(u)|\le C_Uu$ on $(0,U]$, then
\[
 |B_L|\le C_U\frac LM,\qquad
 |B_M|\le\frac{1+(C_U/2)u^2}{M}
          \le\frac{1+(C_U/2)U^2}{M}.
\]
Since $M\ge m/2$ and $L\le d$ throughout $\mathcal R$, there is a
constant $C$, independent of the varying parameters, such that
\[
 \sup_{\mathcal R}|B_M|\le\frac Cm,\qquad
 \sup_{\mathcal R}|B_L|\le\frac{Cd}{m}.
\]

Here is the integration step in detail. Define the one-variable function
\[
 b(s)=B(m-sE,d-sa),\qquad 0\le s\le1.
\]
The chain rule differentiates both coordinates:
\[
 b'(s)=-E B_M(m-sE,d-sa)-a B_L(m-sE,d-sa).
\]
The fundamental theorem of calculus therefore gives the exact identity
\begin{align*}
 B(m-E,d-a)-B(m,d)
 &=b(1)-b(0)\\
 &=-\int_0^1\bigl(E B_M(m-sE,d-sa)
                 +a B_L(m-sE,d-sa)\bigr)\,ds.
\end{align*}
Every point of this segment lies in $\mathcal R$. Taking absolute values,
using the triangle inequality inside the integral, and then the uniform
bounds on $\mathcal R$, we obtain
\begin{align*}
 |B(m-E,d-a)-B(m,d)|
 &\le\int_0^1\bigl(E|B_M(m-sE,d-sa)|
                     +a|B_L(m-sE,d-sa)|\bigr)\,ds\\
 &\le E\sup_{\mathcal R}|B_M|+a\sup_{\mathcal R}|B_L|\\
 &\le C\left(\frac Em+\frac{ad}{m}\right).
\end{align*}
The factors $E$ and $a$ are the coordinate displacements supplied by
the chain rule; the interval of integration has length one.
Finally, $E\le3t^2$, $a\le2$, $m\asymp tk^2$, and $d\le k$ give
\[
 \frac Em=O(t/k^2),\qquad \frac{ad}{m}=O(1/(tk)).
\]
Both are $O(k^{-1})$ because $1\le t\le k$. Thus the entire logarithmic
prefactor changes by $O(k^{-1})$, uniformly. Equivalently, the ratio of
the two prefactors is
\[
 \exp\bigl(B(m-E,d-a)-B(m,d)\bigr)=1+O(k^{-1}).
\]

Finally, choose a fixed $0<\delta<1/6$ in
\eqref{eq:canfield}. At either integer endpoint $(M,L)$ its logarithmic
form is
\[
 \log p(M,\le L)=\Phi(M,L)+B(M,L)
       +O(M^{-1/6+\delta}+L^{-1}).
\]
Both endpoints have $M\ge c_0tk^2/4$ and $L\ge\eta k/2$, so this
last error is $O(k^{-1/3+2\delta}+k^{-1})=o(1)$ uniformly.
Subtracting the two logarithmic formulas therefore gives
\[
 \log\frac{p(m-E,\le d-a)}{p(m,\le d)}
 =a\log q-\frac{Ev}{d}+o(1).
\]
Exponentiating, and using $e^{o(1)}=1+o(1)$ uniformly, proves
\begin{equation}
\label{eq:cylinder-ratio}
 \frac{p(m-E,\le d-a)}{p(m,\le d)}
 =q^a e^{-Ev/d}(1+o(1)).
\end{equation}

\emph{Summing the terms and controlling subtraction.}
The bound \eqref{eq:v-global-bound}, evaluated at $u=d/\sqrt m$,
also gives $v\le d^2/m=O(1/t)$.
Since $|E_i-E_0|=i-1\le t+1$, we have
\[
 \frac{|E_i-E_0|v}{d}=O(k^{-1}),\qquad
 e^{-E_iv/d}=e^{-E_0v/d}(1+O(k^{-1}))
\]
uniformly in $i$. Applying \eqref{eq:cylinder-ratio} to the exact count
thus gives
\[
 \frac{|\mathcal F_{n,k,t}|}{p(m,\le d)}
 =e^{-E_0v/d}\bigl(A_1-A_2+O(\delta_k(A_1+A_2))\bigr),
 \qquad
 A_1=(t+2)q,\quad A_2=(t+1)q^2,
\]
where $\delta_k\to0$ uniformly. To turn this into a relative estimate,
observe that $v\le d^2/m\le2/c_0$, and hence
\[
 1-q=e^{-v}\ge e^{-2/c_0}>0,\qquad
 A_1-A_2\ge A_1(1-q),\qquad A_1+A_2\le2A_1.
\]
It follows that the error divided by $A_1-A_2$ is
$O(\delta_k/(1-q))=o(1)$ uniformly. Since $p(m,\le d)=p(n-t,k-t)$,
this proves \eqref{eq:uniform-size-ratio}.
\end{proof}

We now evaluate the limiting size ratio for fixed $t$ and for growing
$t$ with $t/k\to\beta\in[0,1)$.

\begin{corollary}[Quantitative limits]
\label{cor:quantitative-limits}
The construction has the following limits.
\begin{enumerate}
\item For fixed $t\ge1$ and $c>0$, at $n=\lfloor ctk^2\rfloor$,
\begin{equation}
\label{eq:fixed-t-limit}
 \frac{|\mathcal F_{n,k,t}|}{p(n-t,k-t)}
 \longrightarrow (t+2)q-(t+1)q^2,
 \qquad ct=\frac{I(v)}{v^2},\quad q=1-e^{-v}.
\end{equation}
The limit exceeds one precisely when
\begin{equation}
\label{eq:fixed-t-critical}
 c<c_t:=\frac{I(\log(1+1/t))}{t\log^2(1+1/t)},
 \qquad c_t>1.
\end{equation}
\item If $t\to\infty$, $t/k\to\beta\in[0,1)$, and
$n=\lfloor ctk^2\rfloor$ for fixed $c>0$, then
\begin{equation}
\label{eq:growing-t-limit}
 \frac{|\mathcal F_{n,k,t}|}{p(n-t,k-t)}
 \longrightarrow
 \frac{(1-\beta)^2}{c}
 \exp\left(-\frac{\beta(1-\beta)}{2c}\right).
\end{equation}
In particular, the limit is $1/c$ when $t\to\infty$ and $t=o(k)$.
\end{enumerate}
\end{corollary}

\begin{proof}
We use $v_k,q_k$ for the parameters in Lemma~\ref{lem:uniform-size-ratio}
at the current value of $k$, to distinguish them from their eventual
limits. In both parts,
\[
 m=\lfloor ctk^2\rfloor-k=ctk^2-k+O(1),\qquad
 \frac{m}{tk^2}\longrightarrow c.
\]
For fixed $t$, $d/k\to1$; in the second part, $d/k\to1-\beta>0$.
Thus $d\ge\eta k$ eventually for some fixed $\eta>0$ in each case,
so Lemma~\ref{lem:uniform-size-ratio} applies. Its condition on $c$ is
met, for example, by taking $c_0=c/2$ and $c_1=2c$.

\emph{Fixed $t$.}
Here $m/d^2\to ct$. The function $F(v)=I(v)/v^2$ is continuous and
strictly decreasing from infinity to zero, and hence has a continuous
inverse. Since $F(v_k)=m/d^2$, we obtain
\[
 v_k\longrightarrow v,\qquad F(v)=ct,\qquad
 q_k=1-e^{-v_k}\longrightarrow q=1-e^{-v}.
\]
Also $E_0$ is fixed and $d\to\infty$, so $v_kE_0/d\to0$ and
$e^{-v_kE_0/d}\to1$. Taking limits in
\eqref{eq:uniform-size-ratio} proves \eqref{eq:fixed-t-limit}.

To decide when the limiting ratio exceeds one, factor its difference
from one:
\[
 (t+2)q-(t+1)q^2-1=(1-q)((t+1)q-1).
\]
Since $0<q<1$, this is positive exactly when $q>1/(t+1)$.
Equivalently,
\[
 1-e^{-v}>\frac1{t+1}
 \quad\Longleftrightarrow\quad
 v>v_*:=\log(1+1/t).
\]
As $F$ is strictly decreasing, the latter condition is equivalent to
\[
 ct=F(v)<F(v_*)=t c_t,
\]
which proves the criterion $c<c_t$ in \eqref{eq:fixed-t-critical}.
At equality the limiting ratio is one; this limit alone makes no
assertion about the finite-$k$ comparison.

Finally, $\rho(x)=x/(e^x-1)$ is strictly decreasing, so
$\rho(x)>\rho(v_*)$ for $0\le x<v_*$. Integrating gives
\[
 I(v_*)=\int_0^{v_*}\rho(x)\,dx
 >v_*\rho(v_*)=\frac{v_*^2}{e^{v_*}-1}=t v_*^2.
\]
Dividing by $t v_*^2$ proves $c_t>1$.

\emph{Growing $t$.}
Now put $u_k=d/\sqrt m$. Since $d\le k$ and $m\sim ctk^2$,
\[
 u_k^2=\frac{d^2}{m}=O(1/t)\longrightarrow0.
\]
The expansions $v(u)=u^2(1+O(u^2))$ and $q(u)=u^2(1+O(u^2))$
in \eqref{eq:small-u-asymptotics} imply
\[
 v_k=\frac{d^2}{m}(1+O(1/t)),\qquad
 q_k=\frac{d^2}{m}(1+O(1/t)).
\]
In particular, $q_k=O(1/t)$. The polynomial factor in
\eqref{eq:uniform-size-ratio} can therefore be written as
\begin{align*}
 (t+2)q_k-(t+1)q_k^2
 &=tq_k\left(1+\frac2t-\left(1+\frac1t\right)q_k\right)\\
 &=\frac{td^2}{m}(1+O(1/t)).
\end{align*}
For the exponential factor, use
$E_0=(t^2/2)(1+3/t+2/t^2)$ to obtain
\[
 \frac{v_kE_0}{d}
 =\frac{t^2d}{2m}(1+O(1/t))
 =\frac{t^2d}{2m}+o(1).
\]
The last step is justified by $t^2d/(2m)=O(t/k)=O(1)$.
This boundedness turns the relative error in the exponent into an
additive $o(1)$, and hence
\[
 e^{-v_kE_0/d}
 =e^{-t^2d/(2m)}(1+o(1)).
\]
Combining these estimates with \eqref{eq:uniform-size-ratio} gives
\begin{equation}
\label{eq:growing-t-master}
 \frac{|\mathcal F_{n,k,t}|}{p(n-t,k-t)}
 =(1+o(1))\frac{td^2}{m}
 \exp\left(-\frac{t^2d}{2m}\right).
\end{equation}
To evaluate the two remaining factors separately, write
\[
 \frac{td^2}{m}
 =\frac{(d/k)^2}{m/(tk^2)}
 \longrightarrow\frac{(1-\beta)^2}{c},\qquad
 \frac{t^2d}{2m}
 =\frac{(t/k)(d/k)}{2m/(tk^2)}
 \longrightarrow\frac{\beta(1-\beta)}{2c}.
\]
This proves \eqref{eq:growing-t-limit}. If $t\to\infty$ and $t=o(k)$,
then $\beta=0$, so the exponential tends to one and the limit is $1/c$.
\end{proof}

Finally, we use Corollary~\ref{cor:quantitative-limits} to prove
Theorem~\ref{thm:every-scale}, obtaining counterexamples to Borg's
conjecture at every intersection scale.

\begin{proof}[Proof of Theorem~\ref{thm:every-scale}]
Lemma~\ref{lem:construction-size} already proves that the family is
$t$-intersecting. We must show that, with one threshold for $k$ valid
simultaneously for every $1\le t\le\lceil k/4\rceil$, its size is
strictly greater than the canonical star at $n=\lfloor tk^2/3\rfloor$.

Suppose this uniform assertion were false. We could then choose
integers $k_j\to\infty$ and $1\le t_j\le\lceil k_j/4\rceil$ such that,
with $n_j=\lfloor t_jk_j^2/3\rfloor$,
\[
 R_j:=\frac{|\mathcal F_{n_j,k_j,t_j}|}
                 {p(n_j-t_j,k_j-t_j)}\le1
 \qquad\text{for every }j.
\]
Every sequence of positive integers $(t_j)$ has either a bounded
subsequence or a subsequence tending to infinity. We examine these
two possibilities, retaining the notation $j$ after taking subsequences.

If $(t_j)$ has a bounded subsequence, one integer value $t_*$ occurs
infinitely often. Restrict to $t_j=t_*$. Part~(1) of
Corollary~\ref{cor:quantitative-limits}, with $c=1/3$, gives a limit
$R_j\to R_*>1$, since $1/3<1<c_{t_*}$. This contradicts $R_j\le1$;
indeed convergence makes $R_j>(1+R_*)/2>1$ for all sufficiently large $j$.

Otherwise, pass to a subsequence with $t_j\to\infty$. The ratios
$t_j/k_j$ are bounded, so a further subsequence converges to some
$\beta$. The ceiling in the parameter range causes no problem:
\[
 0\le\frac{t_j}{k_j}\le\frac{\lceil k_j/4\rceil}{k_j}
 \le\frac14+\frac1{k_j},
\]
and consequently $\beta\in[0,1/4]$. Part~(2) of
Corollary~\ref{cor:quantitative-limits} now gives
\[
 R_j\longrightarrow G(\beta)
 :=3(1-\beta)^2e^{-3\beta(1-\beta)/2}.
\]
We check that every possible limiting value exceeds one. The function
$G$ is positive, and
\[
 \frac{d}{d\beta}\log G(\beta)
 =-\frac2{1-\beta}-\frac32(1-2\beta)<0
 \qquad(0\le\beta\le1/4).
\]
Thus $G$ is decreasing on this interval, and its minimum occurs at
$\beta=1/4$. Using $e^{-x}\ge1-x$ for $x\ge0$, we obtain
\[
 G(\beta)\ge G(1/4)=\frac{27}{16}e^{-9/32}
 \ge\frac{27}{16}\left(1-\frac9{32}\right)
 =\frac{621}{512}>1.
\]
This again contradicts $R_j\le1$. Since both possibilities lead to
contradictions, such a sequence of failures cannot exist. Consequently,
there is a single $k_0$ for which the strict size inequality holds for
all $k\ge k_0$ and all allowed $t$.
\end{proof}

\begin{remark}
The construction $\mathcal F_{n,k,t}$ is not always optimal. For
$(n,k,t)=(31,8,1)$, it has $620$ members, whereas the following
intersecting family has $626$ members. Recall that $\mathcal S_Q(n,k)$
denotes the star of all partitions containing the multiset $Q$. Define
\[
 \cH=
 \{\lambda\in P_{31,8}:1\in\lambda,\ 
          \operatorname{supp}(\lambda)\cap\{2,3,4\}\ne\varnothing\}
 \ \cup\ \mathcal S_{\{2,3,4\}}(31,8).
\]
Two members of the first branch share $1$;
two members of the second share $2,3,4$; and a member from each branch
shares at least one value in $\{2,3,4\}$. Inclusion--exclusion, with the overlap of the branches cancelled, gives
\begin{align*}
 |\mathcal H|
 &=p(28,6)+p(27,6)+p(26,6)\\
 &\quad-p(25,5)-p(24,5)-p(23,5)+p(22,5)\\
 &=391+331+282-192-164-141+119=626,
\end{align*}
which exceeds $|\mathcal F_{31,8,1}|=620$.
\end{remark}
In the next section we use the construction above to provide counterexamples to conjectures of Borg on  intersecting unrestricted integer partitions.

\section{Counterexamples for  \texorpdfstring{$t$}{t}-intersecting unrestricted integer partitions}\label{sec:counterexamples_unrestricted}
\subsection{Small counterexamples}
First we present a small counterexample to Borg's conjecture~1.1~\cite{Borg2016}. 

For $n\ge1$ and $\mathcal{H}=\bigl\{\{1,2\},\{1,3\},\{1,4\}, \{2,3,4\}\bigr\}$, define
\begin{equation}\label{eq:unrestrited_family}
 \An=\{\lambda\in P_n:
             H\subseteq\supp(\lambda)\text{ for some }H\in\mathcal H\}.
\end{equation}
Set $p(0)=1$ and $p(m)=0$ for $m<0$. 
\begin{proposition}\label{prop:finite}
The family $\mathcal A_n$ is intersecting, and
\begin{equation}\label{eq:count}
 |\cA_n|=p(n-3)+p(n-4)+p(n-5)
                 -p(n-6)-p(n-7)-p(n-8)+p(n-9).
\end{equation}
In particular, $|\cA_{57}|=527254>526823=|\cT_{57,1}|=p(56)$.
\end{proposition}
\begin{proof}
The four members of $\mathcal H$ are pairwise intersecting, so any two
members of $\mathcal{A}_n$ share a part and, therefore, $\cA_n$ is $1$-intersecting. 

By the inclusion-exclusion formula we thus obtain
\[
 |\cA_n|=p(n-3)+p(n-4)+p(n-5)-p(n-6)-p(n-7)-p(n-8)+p(n-9).
\]

At $n=57$, the required integer partition numbers give
\[
\begin{aligned}
 |\cA_{57}|
 &=386155+329931+281589\\
 &\quad-239943-204226-173525+147273\\
 &=527254>526823=p(56).
\end{aligned}
\]
These integers can be checked by the recurrence
$p_j(m)=p_{j-1}(m)+p_j(m-j)$, where $p_j(m)$ counts partitions into
parts at most $j$, with $p_0(0)=1$, $p_0(m)=0$ for $m>0$, and
$p_j(m)=0$ for $m<0$; then $p(m)=p_m(m)$ for $m\ge1$.
\end{proof}

For a partition $\lambda$ of some positive integer $d$ we write $\lambda\uplus1^r$ to mean the partition of $d+r$ obtained by adding
$r$ ones to $\lambda$; this is a multiset operation.

\begin{corollary}\label{cor:finite-t}
For every integer $t\ge1$, the family
\[
 \{\lambda\uplus1^{t-1}:\lambda\in\mathcal A_{57}\}
 \subseteq P_{t+56}
\]
is $t$-intersecting and has $527254$ members, exceeding the size
$p(56)=526823$ of the canonical $t$-star $\cT_{t+56,t}$.
\end{corollary}

\begin{proof}
Adding the same $t-1$ ones to two partitions increases their number of
common parts by exactly $t-1$:
\begin{equation}\label{eq:lift-intersection}
 I(\lambda\uplus1^{t-1},\mu\uplus1^{t-1})
 =(t-1)+I(\lambda,\mu)\ge t.
\end{equation}
The operation is injective, and $|\cT_{t+56,t}|=p(56)$.
\end{proof}

\subsection{Useful estimates of the partition function}
Next we will analyze the family $\cA_n$ and show that, for large $n$, it has more elements that the canonical star $\cT_{n,1}$ and, by a similar lifting as in the Corollary~\ref{cor:finite-t}, also provide counterexamples to Conjecture~1.4 of Borg~\cite{Borg2016} about largest $t$-intersecting families of unrestricted integer partitions.

For this, we will use estimates on the error bound in the asymptotic expansion of the partition function as given by Nemes~\cite{Nemes2024}. 

Set
\[
 \alpha=\frac{\pi}{\sqrt6},\qquad
 H(n)=\frac{e^{2\alpha\sqrt n}}{4\sqrt3\,n},\qquad
 E(n)=e^{-\alpha\sqrt n}\quad(n\ge1).
\]
Then define, for each integer $r\ge0$,
\begin{equation}\label{eq:coefficients}
 c_r=\frac{(-1)^r}{(4\sqrt6)^r}
 \sum_{k=0}^{\lfloor(r+1)/2\rfloor}
 \binom{r+1}{k}\frac{r+1-k}{(r+1-2k)!}
 \left(\frac{\pi}{6}\right)^{r-2k}.
\end{equation}
In particular,
\[
 c_0=1,\qquad
 c_1=-\frac{\pi^2+72}{24\pi\sqrt6},\qquad
 c_2=\frac1{16}+\frac{\pi^2}{6912}.
\]
Every summand in~\eqref{eq:coefficients} is positive before the factor
$(-1)^r$ is applied, so $(-1)^r c_r>0$.
For integers $N\ge0$ and $n\ge1$, the \emph{normalized remainder}
$R_N(n)$ is defined by
\begin{equation}\label{eq:remainder-definition}
 p(n)=H(n)\left(\sum_{r=0}^{N-1}\frac{c_r}{n^{r/2}}+R_N(n)\right),
\end{equation}
with the sum interpreted as empty when $N=0$.

We will be using the following explicit estimates given by Nemes~\cite[Theorem~1.1]{Nemes2024}.

\begin{theorem}\label{thm:nemes}
For all integers $k\ge0$ and $n\ge1$,
\begin{equation}\label{eq:even-remainder}
 -E(n)<R_{2k}(n)<\frac{c_{2k}}{n^k}+E(n),
\end{equation}
and
\begin{equation}\label{eq:odd-remainder}
 \frac{c_{2k+1}}{n^{k+1/2}}-E(n)<R_{2k+1}(n)<E(n).
\end{equation}
\end{theorem}

The estimates are asymmetric. In particular, taking $k=1$
in~\eqref{eq:odd-remainder}, and using $c_3<0$, gives
\begin{equation}\label{eq:R3}
 \frac{c_3}{n^{3/2}}-E(n)<R_3(n)<E(n),
 \qquad
 |R_3(n)|\le\frac{|c_3|}{n^{3/2}}+E(n)=O(n^{-3/2}).
\end{equation}

The following lemma is a straightforward consequence of the above theorem and will allow for comparison of the partition functions between values which differ in a constant. 

\begin{lemma}\label{lem:ratio}
For each fixed integer $j$, as $n\to\infty$,
\begin{equation}\label{eq:ratio}
 \frac{p(n-j)}{p(n)}
 =1-\frac{\alpha j}{\sqrt n}
  +\frac{j+\alpha^2j^2/2}{n}+O(n^{-3/2}).
\end{equation}
The error constant can be chosen uniformly when $j$ ranges over any
fixed finite set of integers.
\end{lemma}

\begin{proof}
Write $S(n)=p(n)/H(n)$. By~\eqref{eq:remainder-definition}
and~\eqref{eq:R3},
\[
 S(n)=1+\frac{c_1}{\sqrt n}+\frac{c_2}{n}+R_3(n),
 \qquad R_3(n)=O(n^{-3/2}).
\]
For fixed $j$, subtracting the expressions at $n-j$ and $n$ yields
\[
\begin{split}
 S(n-j)-S(n)={}&c_1\big((n-j)^{-1/2}-n^{-1/2}\big)\\
 &+c_2\big((n-j)^{-1}-n^{-1}\big)
   +R_3(n-j)-R_3(n)=O(n^{-3/2}).
\end{split}
\]
The first  two explicit differences are $O(n^{-3/2})$ and $O(n^{-2})$,
respectively. The remainder difference is bounded by the sum of its
two absolute values, hence is $O(n^{-3/2})$ as well. Since $S(n)\to1$, division gives
\begin{equation}\label{eq:S-ratio}
 \frac{S(n-j)}{S(n)}=1+O(n^{-3/2}).
\end{equation}
On the other hand,
\[
 \frac{H(n-j)}{H(n)}
 =\frac{n}{n-j}
  \exp\bigl(2\alpha(\sqrt{n-j}-\sqrt n)\bigr),
\]
which we rewrite as 
\[
 \ln\left(\frac{H(n-j)}{H(n)}\right)
 =-\ln(1-j/n)+
  2\alpha(\sqrt{n-j}-\sqrt n).
\]

For fixed $j$,  using Taylor expansion gives
\[
-\ln(1-j/n)+
  2\alpha(\sqrt{n-j}-\sqrt n)=-\frac{\alpha j}{\sqrt{n}}+\frac{j}{n}+O(n^{-3/2}).
\]
Exponentiating we obtain
\[
\begin{aligned}
\frac{H(n-j)}{H(n)}
&=\exp\left(-\frac{\alpha j}{\sqrt n}
             +\frac jn+O(n^{-3/2})\right)\\
&=1-\frac{\alpha j}{\sqrt n}
  +\frac{j+\alpha^2j^2/2}{n}+O(n^{-3/2}).
\end{aligned}
\]
We thus obtain
\[
\begin{split}
\frac{p(n-j)}{p(n)}&=\frac{S(n-j)}{S(n)}\frac{H(n-j)}{H(n)}
=\\
&\left(1+O(n^{-3/2})\right)\left(1-\frac{\alpha j}{\sqrt{n}}+\frac{j+\alpha^2j^2/2}{n}+O(n^{-3/2})\right),
\end{split}
\]
which is the claimed $1-\frac{\alpha j}{\sqrt{n}}+\frac{j+\alpha^2j^2/2}{n}+O(n^{-3/2})$.
\end{proof}

\subsection{Large counterexamples}
Lemma~\ref{lem:ratio} yields the following family of counterexamples for large $1$-intersecting unrestricted integer partitions. 
\begin{proposition}\label{prop:asymptotic}
The family $\cA_n$ satisfies
\begin{equation}\label{eq:A-asymptotic}
 \frac{|\cA_n|}{p(n)}
 =1-\frac{9\alpha^2}{n}+O(n^{-3/2}).
\end{equation}
Moreover,
\begin{equation}\label{eq:gain-asymptotic}
 \frac{|\cA_n|-p(n-1)}{p(n)}
 =\frac{\alpha}{\sqrt n}
  -\frac{1+19\alpha^2/2}{n}+O(n^{-3/2}).
\end{equation}
In particular, $|\cA_n|>|\cT_{n,1}|=p(n-1)$ for all sufficiently large $n$.
\end{proposition}

\begin{proof}
By~\eqref{eq:count} from Proposition~\ref{prop:finite}, we have
 \[
 |\cA_n|=p(n-3)+p(n-4)+p(n-5)-p(n-6)-p(n-7)-p(n-8)+p(n-9).
 \]
Let $a_3=a_4=a_5=a_9=1$ and $a_6=a_7=a_8=-1$. 
The three sums controlling the constant, $n^{-1/2}$, and $n^{-1}$ terms in $\frac{|\cA_n|}{p(n)}$ via~\eqref{eq:ratio} are
\[
\begin{aligned}
 \sum_j a_j&=4-3=1,\\
 \sum_j ja_j&=3+4+5+9-6-7-8=0,\\
 \sum_j j^2a_j&=3^2+4^2+5^2+9^2-6^2-7^2-8^2=-18.
\end{aligned}
\]
Applying Lemma~\ref{lem:ratio} to this fixed set of shifts yields
\[
 \frac{|\mathcal A_n|}{p(n)}
 =1-\frac{9\alpha^2}{n}+O(n^{-3/2}).
\]
Since $9\alpha^2=3\pi^2/2$, this is~\eqref{eq:A-asymptotic}.
Applying the same lemma with $j=1$ gives
\[
 \frac{p(n-1)}{p(n)}
 =1-\frac{\alpha}{\sqrt n}
   +\frac{1+\alpha^2/2}{n}+O(n^{-3/2}).
\]
Subtracting proves~\eqref{eq:gain-asymptotic}. Its leading term is
positive and dominates the $O(n^{-1})$ correction. Thus the gain is
positive for all sufficiently large $n$; in particular,
\[
 |\mathcal A_n|-p(n-1)\sim\frac{\pi}{\sqrt{6n}}\,p(n).
\]
\end{proof}

Finally, we lift our construction again to make families $t$ intersecting as follows. For integers $1\le t\le n$, put $d=n-t$ and define
\begin{equation}\label{eq:B-definition}
 \cA_{n,t}
 =\{\lambda\uplus1^{t-1}:\lambda\in\mathcal A_{d+1}\}
 \subseteq P_n.
\end{equation}
Equation~\eqref{eq:lift-intersection} shows that $\cA_{n,t}$ is
$t$-intersecting and $|\cA_{n,t}|=|\cA_{d+1}|$. Consequently,
\begin{equation}\label{eq:B-count}
\begin{split}
 |\cA_{n,t}|=|\cA_{d+1}|
 ={}&p(d-2)+p(d-3)+p(d-4)\\
    &-p(d-5)-p(d-6)-p(d-7)+p(d-8).
\end{split}
\end{equation}
Thus its size, and its comparison with the canonical star of size
$|\cT_{n,t}|=p(d)$, depend exactly on the gap $d=n-t$.

\begin{theorem}\label{thm:uniform-t}
As $d=n-t\to\infty$,
\begin{equation}\label{eq:B-asymptotic}
 \frac{|\cA_{n,t}|}{p(n-t)}
 =1+\frac{\alpha}{\sqrt d}
     -\frac{1+17\alpha^2/2}{d}+O(d^{-3/2}),
 \qquad\text{where  } \alpha=\frac{\pi}{\sqrt6},
\end{equation}
with an absolute error constant independent of $n$ and $t$.
In particular, there is an absolute integer $D_0$ such that for every
pair of integers $n,t$ with $t\ge1$ and $n-t\ge D_0$,
\[
 |\cA_{n,t}|>|\cT_{n,t}|=p(n-t).
\]
\end{theorem}

\begin{proof}
Use Lemma~\ref{lem:ratio} with $j=-1$ and $n=d$ to obtain
\[
 \frac{p(d+1)}{p(d)}
 =1+\frac{\alpha}{\sqrt d}
    +\frac{-1+\alpha^2/2}{d}+O(d^{-3/2}).
\]
By Proposition~\ref{prop:asymptotic},
\[
 \frac{|\cA_{d+1}|}{p(d+1)}
 =1-\frac{9\alpha^2}{d}+O(d^{-3/2}).
\]
Multiplying these two expressions proves~\eqref{eq:B-asymptotic}.
The only asymptotic variable is $d$; no error term depends on $t$.
Finally (recall that $d=n-t$),
\[
 \frac{|\cA_{n,t}|-p(d)}{p(d)}
 =\frac{\alpha}{\sqrt d}+O(d^{-1})>0
\]
for every sufficiently large $d$, with an absolute threshold.
\end{proof}

The theorem allows $t$ to grow with $n$. In particular, for any
sequence of pairs $(n,t)$ with $n-t\to\infty$, these families eventually
give counterexamples, even when $t/n\to1$.

\begin{remark}[The explicit cutoff]
GPT-6.1 Sol claims that one can take $d=n-t\ge56$ and $n\ge 57$, but it amounts to some elementary recurrence checks and doesn't appear worth writing out in full. 
\end{remark}

\section{Concluding remarks}\label{sec:conclude}
Our counterexamples show that any star bound valid uniformly across all
intersection scales must allow a threshold of order at least $tk^2$.
Theorem~\ref{thm:main-upper} establishes such a bound under the stronger
assumption $n\ge Ak^3$. In the first version of ArXiV-submission we conjectured that a sufficiently large constant multiple of $tk^2$ already would suffice and claimed a proof generated by GPT-6 Astra to require $n\ge Ctk^2$ for all $\lceil(\log k)^2\rceil\le t<k$. 

After several rounds of conversations with Claude (Opus~5.5 and Fable~5.1) and GPT-6.1 Sol and GPT-6 Astra we obtained the proof of the full conjecture, which will be given in a separate paper.

\begin{theorem}\label{thm:conj_main}
There are absolute constants $C,k_0$ such that, whenever $k\ge k_0$,
$1\le t<k$ and $n\ge Ctk^2$, every
$t$-intersecting $\cF\subseteq P_{n,k}$ satisfies
\[
 |\cF|\le p(n-t,k-t).
\]
Equality holds precisely for full $t$-stars of maximum size.
When $k-t\ge3$, the full star of $t$ ones is the unique such star.
\end{theorem}

\bibliographystyle{amsplain}
\bibliography{references}

\end{document}